\documentclass[10pt]{amsart}

\def\bR {\mathbf{R}}
\def\bS {\mathbf{S}}
\def\bT {\mathbf{T}}

\def\cB {\mathcal{B}}

\def\cD {\mathcal{D}}

\def\cG {\mathcal{G}}
\def\cH {\mathcal{H}}

\def\cK {\mathcal{K}}
\def\cL {\mathcal{L}}
\def\cM {\mathcal{M}}

\def\cP {\mathcal{P}}

\def\cR {\mathcal{R}}

\def\a {{\alpha}}
\def\b {{\beta}}

\def\Ga {{\Gamma}}
\def\de {{\delta}}
\def\eps {{\epsilon}}
\def\th {{\theta}}

\def\l {{\lambda}}
\def\L {{\Lambda}}
\def\si {{\sigma}}
\def\Si {{\Sigma}}

\def\om {{\omega}}
\def\Om {{\Omega}}

\def\d {{\partial}}
\def\grad {{\nabla}}
\def\Dlt {{\Delta}}

\def\rstr {{\big |}}

\def\la {\langle}
\def\ra {\rangle}
\def \La {\bigg\langle}
\def \Ra {\bigg\rangle}

\def\Ma{{\hbox{M\!a}}}
\def\Rey{{\hbox{R\!e}}}

\newcommand{\Div}{\operatorname{div}}

\newcommand{\Sign}{\operatorname{sign}}

\newcommand{\Supp}{\operatorname{supp}}

\newcommand{\Dist}{\operatorname{dist}}

\newcommand{\ba}{\begin{aligned}}
\newcommand{\ea}{\end{aligned}}

\newcommand{\be}{\begin{equation}}
\newcommand{\ee}{\end{equation}}

\newcommand{\lb}{\label}

\newtheorem{Thm}{Theorem}[section]

\newtheorem{Prop}[Thm]{Proposition}

\newtheorem{Lem}[Thm]{Lemma}

\begin{document}

\title[From Boltzmann to Euler with Boundaries]{From the Boltzmann Equation\\ to the Euler Equations\\ in the Presence of Boundaries}

\author[F. Golse]{Fran\c cois Golse}
\address[F.G.]{Ecole Polytechnique, Centre de Math\'ematiques L. Schwartz, 91128 Palaiseau Cedex France \& Universit\'e Paris-Diderot, Laboratoire 
J.-L. Lions, BP187, 4 place Jussieu, 75252 Paris Cedex 05 France}
\email{francois.golse@math.polytechnique.fr}

\begin{abstract}
The fluid dynamic limit of the Boltzmann equation leading to the Euler equations for an incompressible fluid with constant density in the presence 
of material boundaries shares some important features with the better known inviscid limit of the Navier-Stokes equations. The present paper slightly 
extends recent results from [C. Bardos, F. Golse, L. Paillard, Comm. Math. Sci., \textbf{10} (2012), 159--190] to the case of boundary conditions for the 
Boltzmann equation more general than Maxwell's accomodation condition.
\end{abstract}



\maketitle


\section{The Inviscid Limit of the Navier-Stokes Equations with Dirichlet Condition}\lb{S-NS>E}


Consider the flow of an incompressible fluid with constant density confined in a domain $\Om\subset\bR^N$ (in practice, $N=2$ or $3$), with smooth
boundary $\d\Om$. The dimensionless form of the Navier-Stokes equations governing the velocity field of the fluid $u_\eps\equiv u_\eps(t,x)\in\bR^N$ 
is
\be\lb{NSDir}
\left\{
\ba
{}&\Div_xu_\eps=0\,,
\\
&\d_tu_\eps+\Div_x(u_\eps\otimes u_\eps)+\grad_xp_\eps=\eps\Dlt_xu_\eps\,,\qquad x\in\Om\,,\,\,t>0\,,
\\
&u_\eps\rstr_{\d\Om}=0\,,
\\
&u_\eps\rstr_{t=0}=u^{in}\,,
\ea
\right.
\ee
where $\eps=\hbox{Re}^{-1}$ is the reciprocal Reynolds number of the flow. The Dirichlet boundary condition $u_\eps\rstr_{\d\Om}=0$ means that
the fluid does not slip on the boundary $\d\Om$ of the domain.

An outstanding problem in fluid dynamics is to understand the behavior of $u_\eps$ in the vanishing $\eps$ limit, and especially to decide whether, 
in that limit, $u_\eps$ converges to the solution of the Euler equations
\be\lb{Eul}
\left\{
\ba
{}&\Div_xu=0\,,
\\
&\d_tu+\Div_x(u\otimes u)+\grad_xp=0\,,\qquad x\in\Om\,,\,\,t>0\,,
\\
&u\cdot n\rstr_{\d\Om}=0\,,
\\
&u\rstr_{t=0}=u^{in}\,,
\ea
\right.
\ee
where $n$ is the unit outward normal field on $\d\Om$.

In the case of the Cauchy problem set in either the whole Euclidian space $\bR^N$ or the periodic box $\bT^N$, i.e. in the absence of material 
boundaries, the Navier-Stokes solutions converge to the solution of the Euler equations as long as the latter remains smooth --- see for instance 
Propositions 4.1 and 4.2 in \cite{LionsBook1}.  We recall that, if $N=2$ and $u^{in}\in C^{k,\a}$ (i.e. either $C^{k,\a}(\bR^N)$ or $C^{k,\a}(\bT^N)$) 
for some $k\ge 1$ and $\a\in(0,1)$, there exists a unique \textit{global} solution $u\in C(\bR_+;C^{k,\a})$ of the Euler equations --- see for instance 
Theorem 4.1 in \cite{LionsBook1} --- while, if $N=3$ and if $u^{in}\in C^{1,\a}$, there exists a unique \textit{local} solution $u\in C([0,T);C^{1,\a})$ 
for some $T\in(0,+\infty]$, and it is unknown at the time of this writing whether $T=+\infty$. Thus in the absence of material boundaries, the inviscid 
limit of the Navier-Stokes equations is described by the Euler equations, globally in time if $N=2$, and maybe only locally in time if $N=3$.

The situation is completely different in a domain with material boundaries. Even when (\ref{Eul}) has a smooth solution, this solution might fail to attract
$u_\eps$ in the limit as $\eps\to 0$. The reason for such a behavior is of course that the Dirichlet condition satisfied by the Navier-Stokes solutions 
$u_\eps$ for each $\eps>0$ is overdetermined for the Euler equations. Indeed, in general, the tangential component of the solution $u$ of the Euler 
equations does not vanish on $\d\Om$. 

Thus one may seek to match the Euler solution to the Dirichlet condition on the boundary with a viscous boundary layer of thickness $\sqrt{\eps}$ as 
proposed by Prandtl --- see for instance \cite{Schlichting}. However it may not always be possible to do so. Consider for instance the flow of a viscous 
fluid past a cylinder or a sphere. In other words, assume that $\Om$ is the complement in $\bR^2$ of a circular cylinder --- or the complement in $\bR^3$ 
of a ball, and that the velocity field is constant at infinity on one side of the immersed body. Although there is no mathematical proof of this fact to this date, 
it is expected, on the basis of experiments and numerical simulations that, already at moderate Reynolds numbers, the viscous boundary layer detaches 
from the boundary and that vortices form in the wake past the immersed body (a phenomenon known as a ``von Karman street''). Yet there exist perfectly
smooth solutions of the Euler equations corresponding to the potential flow of an incompressible fluid past a sphere in space dimension 3 --- see for instance \S 10, Problem 2 in \cite{Landau6}.

Not much is known on Prandtl's boundary layer analysis for this problem from the mathematical viewpoint, apart from some positive results based 
on the Cauchy-Kovalevski theorem \cite{CaflSamm1,CaflSamm2}, as well as negative results concerning the Prandtl boundary layer equations 
\cite{EEngqu,Gren}. There is however the following very interesting criterion due to Kato \cite{KatoBL}.

\begin{Thm}[Kato]
Assume that 
$$
\int_\Om|u^{in}(x)|^2dx<\infty
$$
and denote 
$$
\d\Om_\eps:=\{x\in\Om\,|\,\Dist(x,\d\Om)<\eps\}\,.
$$
Then
$$
\int_0^T\int_{\Om}|u_\eps-u|^2dxdt\to 0\quad\Leftrightarrow\quad\eps\int_0^T\int_{\d\Om_\eps}|\grad_xu_\eps|^2dxdt\to 0
$$
as $\eps\to 0$.
\end{Thm}

In other words, the convergence of the Navier-Stokes solutions $u_\eps$ to $u$ as $\eps\to 0$ in quadratic mean everywhere in $\Om$ is equivalent 
to the vanishing of the viscous energy dissipation in a thin layer near the boundary. The convergence of the Navier-Stokes solutions to a solution of the
Euler equations is therefore a strongly nonlocal phenomenon. Notice that, while the Prandtl viscous boundary layer has thickness $\sqrt{\eps}$, Kato's 
criterion involves the vanishing of viscous dissipation in a much thinner sublayer, of thickness $\eps$. On the other hand, while Prandtl's theory is based 
on the construction of a multiscale asymptotic expansion for $u_\eps$, Kato's result is based on an energy estimate -- so that, in theory, it might happen
that Prandtl's construction breaks down while $u_\eps$ converges to $u$ in quadratic mean. (We are however unaware of examples of such flows.)


\section{The Inviscid Limit of the Navier-Stokes Equations\\ with Slip Boundary Condition}\lb{S-NS2>E}


To confirm the role of the Dirichlet boundary condition as a source of difficulties in the inviscid limit, we supplement the Navier-Stokes equations with
a more general class of boundary conditions, known as ``slip boundary conditions'', which take the form
\be\lb{SlipBC}
u_\eps\cdot n\rstr_{\d\Om}=0\,,\quad\eps(\Si(u_\eps)\cdot n)_\tau+\l u_\eps\rstr_{\d\Om}=0\,,
\ee
where $\l$ is a scalar (the slip coefficient), and where $\Si(u_\eps)$ denotes the deformation tensor. In other words, for each vector field $v$ defined 
on a neighborhood of $\overline\Om$,
\be\lb{DefSi}
\Si(v):=\grad_xv+(\grad_xv)^T\,,
\ee
while the subscript $\tau$ denotes the tangential component on the boundary: 
\be\lb{Def-vtau}
v_\tau:=v-(v\cdot n)n\,.
\ee
Since $u_\eps\cdot n\rstr_{\d\Om}=0$, one has
$$
(\Si(u_\eps)\cdot n)_\tau\rstr_{\d\Om}=\left(\frac{\d u_\eps}{\d n}\right)_\tau\bigg|_{\d\Om}-(\grad n)\cdot u_\eps\bigg|_{\d\Om}
$$
so that $(\Si(u_\eps)\cdot n)_\tau\rstr_{\d\Om}$ and $\left(\frac{\d u_\eps}{\d n}\right)_\tau\bigg|_{\d\Om}$ differ by a lower order term --- of order $0$,
in the sense of differential operators --- involving the Weingarten endomorphism $\grad n$ acting on the tangent space of $\d\Om$.

Henceforth we assume that the slip coefficient $\l$ depends on the Reynolds number $\eps^{-1}$, and denote it by $\l\equiv\l_\eps$. Thus

a) if $\l_\eps\ge\l_0>0$ for all $\eps>0$, one expects that the slip boundary condition (\ref{SlipBC}) should be asymptotically equivalent to the Dirichlet
condition, thereby leading to the same difficulties as regards the inviscid limit;

b) if $\l_\eps=0$, the slip boundary condition (\ref{SlipBC}) reduces to the Navier full slip condition
\be\lb{FullSlip}
u_\eps\cdot n\rstr_{\d\Om}=0\,,\quad(\Si(u_\eps)\cdot n)_\tau\rstr_{\d\Om}=0\,.
\ee
As is well known, the Navier full slip condition prevents the detachment of viscous boundary layers, so that the inviscid limit of the Navier-Stokes 
equations supplemented with that condition is indeed governed by the Euler equations, whenever there exists a smooth solution of these equations --- 
see for instance \cite{BardosEuler} or step 2 in the proof of Theorem 4.1 in \cite{LionsBook1}.

In fact, a more general result is true. Consider the Navier-Stokes equations supplemented with the slip boundary condition
\be\lb{NSSlip}
\left\{
\ba
{}&\Div_xu_\eps=0\,,
\\
&\d_tu_\eps+\Div_x(u_\eps\otimes u_\eps)+\grad_xp_\eps=\eps\Dlt_xu_\eps\,,\quad x\in\Om\,,\,\,t>0\,,
\\
&u_\eps\!\cdot\! n\rstr_{\d\Om}\!=\!0\,,\quad\!\eps(\Si(u_\eps)\!\cdot\! n)_\tau\!+\!\l_\eps u_\eps\rstr_{\d\Om}=0\,,
\\
&u_\eps\rstr_{t=0}=u^{in}\,.
\ea
\right.
\ee
The existence of global weak solutions ``\`a la Leray'' of (\ref{NSSlip}) is classical: see for instance Theorem 2 in \cite{IftiSueur}. 

Henceforth we use the classical notation $\cH(\Om)$ to designate the set of solenoidal vector fields on $\Om$ that are tangent to $\d\Om$ and have
finite square mean in $\Om$, i.e.
$$
\cH(\Om):=\{v\in L^2(\Om)\,|\,\Div v=0\hbox{ and }v\cdot n\rstr_{\d\Om}=0\}\,.
$$

 \begin{Thm}[Bardos-Golse-Paillard \cite{BGP}]\label{T-NS>Eul}
Let $u^{in}\in C^1_c(\Om)$ and for  each $\eps>0$ let $u_\eps$ be a Leray solution of (\ref{NSSlip}). Assume that the Euler equations (\ref{Eul}) 
have a local classical solution $u\in C^1_b([0,T]\times\Om)$ defined for some $T>0$ and satisfying
\be\lb{uH1}
\int_0^T\int_{\d\Om}|u(t,x)|^2dS(x)dt+\int_0^T\int_\Om|\Si(u)(t,x)|^2dxdt<\infty\,,
\ee
where $dS(x)$ is the surface element on $\d\Om$ if $N=3$ (resp. the length element if $N=2$). Then
$$
\l_\eps\to 0\quad\Rightarrow\quad\sup_{0\le t\le T}\int_\Om|u_\eps-u|^2(t,x)dx\to 0\,.
$$
\end{Thm}

The proof will serve as a model in the case of the Boltzmann equation; it is given in detail below.

\begin{proof}
Leray solutions of (\ref{NSSlip}) satisfy the energy inequality
\be\lb{Leray<}
\ba
\tfrac12\int_\Om|u_\eps(t,x)|^2dx+\eps\int_0^t\int_\Om\tfrac12|\Si(u_\eps)(s,x)|^2dxds+\l_\eps\int_0^t\int_{\d\Om}|u_\eps(s,x)|^2dS(x)ds&
\\
\le\tfrac12\int_\Om|u^{in}(x)|^2dx&\,.
\ea
\ee
On the other hand, for each test vector field $w\in C^\infty_c(\bR_+\times\overline\Om)\cap C(\bR_+;\cH(\Om))$
$$
\ba
\int_\Om u_\eps\cdot&w(t,x)dx-\int_\Om u^{in}(x)\cdot w(0,x)dx
\\
&+\eps\int_0^t\int_\Om\tfrac12\Si(u_\eps):\Si(w)(s,x)dxds+\l_\eps\int_0^t\int_{\d\Om}u_\eps\cdot w(s,x)dS(x)ds
\\
&=
\int_0^t\int_\Om u_\eps\cdot E(w)(s,x)dxds+\int_0^t\int_\Om\grad_xw:(u_\eps-w)^{\otimes 2}(s,x)dxds
\ea
$$
where
\be\lb{DefE}
E(w):=\d_tw+w\cdot\grad_xw\,.
\ee
Therefore, with the identity
$$
\frac{d}{dt}\int_\Om\tfrac12|w(t,x)|^2dx=\int_\Om w\cdot E(w)(t,x)dx\,,
$$
we conclude that
\be\lb{DiffEner}
\ba
\tfrac12\int_\Om|u_\eps-w|^2(t,x)dx-\tfrac12\int_\Om|u^{in}(x)-w(0,x)|^2dx&
\\
+\eps\int_0^t\int_\Om\tfrac12|\Si(u_\eps)(s,x)|^2dxds+\l_\eps\int_0^t\int_{\d\Om}|u_\eps(s,x)|^2dS(x)ds&
\\
\le\eps\int_0^t\int_\Om\tfrac12\Si(u_\eps):\Si(w)(s,x)dxds+\l_\eps\int_0^t\int_{\d\Om}u_\eps\cdot w(s,x)dS(x)ds&
\\
-\int_0^t\int_\Om u_\eps\cdot E(w)(s,x)dxds-\int_0^t\int_\Om\grad_xw:(u_\eps-w)^{\otimes 2}(s,x)dxds&\,.
\ea
\ee
By a straightforward density argument, we may replace $w$ with the solution $u$ of (\ref{Eul}): thus
$$
\ba
\int_0^t\int_\Om u_\eps\cdot E(u)(s,x)dxds&=-\int_0^t\int_\Om u_\eps\cdot\grad_xp(s,x)dxds
\\
&=-\int_0^t\int_{\d\Om}pu_\eps(s,x)\cdot n_xdS(x)ds=0\,,
\ea
$$
so that
\be\lb{DiffEner2}
\ba
\tfrac12\int_\Om|u_\eps-u|^2(t,x)dx
\!+\!\eps\int_0^t\int_\Om\tfrac12|\Si(u_\eps)(s,x)|^2dxds\!+\!\l_\eps\int_0^t\int_{\d\Om}|u_\eps(s,x)|^2dS(x)ds&
\\
\le\eps\int_0^t\int_\Om\tfrac12\Si(u_\eps):\Si(u)(s,x)dxds+\l_\eps\int_0^t\int_{\d\Om}u_\eps\cdot u(s,x)dS(x)ds&
\\
+\si(u)\int_0^t\int_\Om|u_\eps-u|^2(s,x)dxds&\,,
\ea
\ee
where
$$
\si(u):=\sup_{0\le t\le T\atop x\in\Om}|\grad_xu(t,x)|\,.
$$
Define
$$
Q_\eps(t):=\eps\int_\Om\tfrac12|\Si(u_\eps)(t,x)||\Si(u)(t,x)|dx+\l_\eps\int_{\d\Om}|u_\eps(t,x)||u(t,x)|dS(x)\,.
$$
By the Leray energy inequality
$$
\eps\int_0^T\int_\Om\tfrac12|\Si(u_\eps)(s,x)|^2dxds+\l_\eps\int_0^T\int_{\d\Om}|u_\eps(s,x)|^2dS(x)ds=O(1)
$$
so that, by (\ref{uH1}) and the Cauchy-Schwarz inequality,
$$
\int_0^TQ_\eps(t)dt=O(\sqrt{\eps})+O(\sqrt{\l_\eps}).
$$
Applying Gronwall's lemma shows that
$$
\int_\Om\tfrac12|u_\eps-u|^2(t,x)dx\le e^{T\si(u)}\int_0^TQ_\eps(s)ds\to 0
$$
as $\eps\to 0$.
\end{proof}


\section{The Incompressible Euler Limit of the Boltzmann Equation}


\subsection{Background on the Boltzmann Equation}

The Boltzmann equation governs the distribution function of molecules in a monatomic gas, denoted here by $F\equiv F(t,x,v)$. We recall that $F(t,x,v)$
is the density with respect to the phase space measure $dxdv$ of gas molecules located at the position $x\in\bR^3$ with velocity $v\in\bR^3$ at time $t$.
It takes the form
$$
(\d_t+v\cdot\grad_x)F=\cB(F,F)
$$
where $\cB(F,F)$ denotes the Boltzmann collision integral. In the case where gas molecules behave like hard spheres subject to elastic, binary collisions, 
the collision integral takes the (dimensionless) form
$$
\cB(F,F):=\iint_{\bR^3\times\bS^2}(F'F'_*-FF_*)|(v-v_*)\cdot\om|dv_*d\om\,.
$$
Here we have used the notation
$$
\left\{
\ba
F&\equiv F(t,x,v)\,,\quad &&F_*\equiv F(t,x,v_*)\,,
\\
F'&\equiv F(t,x,v')\,,\quad &&F'_*\equiv F(t,x,v'_*)\,,
\ea
\right.
$$
where $v'$ and $v'_*$ are defined in terms of $v,v_*\in\bR^3$ and $\om\in\bS^2$ by the formulas
$$
\left\{
\ba
v'&\equiv v'(v,v_*,\om)\,:=v\,-(v-v_*)\cdot\om\om\,,
\\
v'_*&\equiv v'_*(v,v_*,\om):=v_*\!+(v-v_*)\cdot\om\om\,.
\ea
\right.
$$

The Boltzmann collision integral enjoys the following well known properties. 

First it satisfies the local conservation laws of mass, momentum and energy, in the following form: if $F$ is rapidly decaying as $|v|\to\infty$, then 
\be\lb{ConsBCollInt}
\int_{\bR^3}\cB(F,F)\left(\begin{matrix}1\\ v_k\\ \tfrac12|v|^2\end{matrix}\right)dv=0\,,\quad k=1,2,3.
\ee

Another fundamental property of the collision integral is Boltzmann's H Theorem, which can be stated as follows: if $F>0$ is rapidly decaying in $v$ 
while $\ln F=O(|v|^n)$ for some $n\ge 0$, then
\be\lb{HThmCollInt}
\int_{\bR^3}\cB(F,F)\ln Fdv\le 0\,,
\ee
while 
$$
\int_{\bR^3}\cB(F,F)\ln Fdv\!=\!0\Leftrightarrow\cB(F,F)\!=\!0\Leftrightarrow F\!=\!\cM_{(\rho,u,\th)}\,,
$$
where the notation $\cM_{(\rho,u,\th)}$ designates the Maxwellian distribution with density $\rho\ge 0$, bulk velocity $u\in\bR^3$ and temperature 
$\th>0$, i.e.
$$
\cM_{(\rho,u,\th)}(v):=\frac{\rho}{(2\pi\th)^{3/2}}e^{-\frac{|v-u|^2}{2\th}}\,.
$$
We shall formulate below the corresponding differential relations for solutions of the Boltzmann equation.

\subsection{Boundary Conditions for the Boltzmann Equation}

A general class of boundary conditions for the Boltzmann equation is of the form
\be\lb{ScattBC}
F(t,x,v)|v\!\cdot\! n_x|=\int_{v'\!\cdot n_x>0}\!\!F(t,x,v')v'\!\cdot\! n_x\cK(x,\cR_xv,dv')\,, \quad(x,v)\in\Ga_-\,,
\ee
where we recall that the notation $n_x$ designates the unit outward normal at the point $x\in\d\Om$ while $\cR_x$ designates the specular reflection: 
$$
\cR_xv:=v-2v\cdot n_xn_x\,,
$$
and where we have used the notation
$$
\begin{aligned}
\Ga_+:=\{(x,v)\in\d\Om\times\bR^3\,|\,v\cdot n_x>0\}\,,
\\
\Ga_-:=\{(x,v)\in\d\Om\times\bR^3\,|\,v\cdot n_x<0\}\,.
\end{aligned}
$$
The measure-valued, scattering kernel $\cK(x,v,dv')\ge 0$ satisfies the assumptions (see \S 8.2 in \cite{CIP}, especially on pp. 230-231):

\smallskip
\noindent
(i) for each $x\in\d\Om$, one has
$$
\int_{v\cdot n_x>0}\cK(x,v,dv')dv=dv'\,;
$$
(ii) for each $x\in\d\Om$ and each $\Phi\in C_c(\bR^3\times\bR^3)$ one has
$$
\ba
\iint_{\bR^3\times\bR^3}\Phi(v,v')(v\cdot n_x)_+(v'\cdot n_x)_+\cM_{(1,0,\th)}(v)\cM_{(1,0,\th)}(v')dv\cK(x,\cR_xv,dv')
\\
=
\iint_{\bR^3\times\bR^3}\Phi(-v',-v)(v\cdot n_x)_+(v'\cdot n_x)_+\cM_{(1,0,\th)}(v)\cM_{(1,0,\th)}(v')dv\cK(x,\cR_xv,dv')
\ea
$$
(If $\cK(x,v,dv')$ is of the form $\cK(x,v,v')=K(x,v,v')dv'$, property (ii) is equivalent to the identity 
$$
K(x,v,v')=K(x,v',v)\qquad\hbox{ for a.e. }v,v'\hbox{ s.t. }v\cdot n_x>0\hbox{ and }v'\cdot n_x>0
$$
for each $x\in\d\Om$.)

\noindent
(iii) for each $x\in\d\Om$, one has
$$
\cM_{(1,0,\th)}(v)v\!\cdot\! n_x=\int_{v'\!\cdot n_x>0}\cM_{(1,0,\th)}(v')v'\!\cdot\! n_x\cK(x,v,dv')\,,\quad v\!\cdot\! n_x>0\,,
$$
if and only if $\th=\th_w(x)$, where $\th_w(x)$ is the temperature of the boundary $\d\Om$ at the point $x$.

With these notations, the case of diffuse reflection corresponds to
$$
\cK(x,v,dv'):=\frac{\cM_{(1,0,\th_w)}(v)|v\cdot n_x|dv'}{\displaystyle\int_{u\cdot n_x<0}\cM_{(1,0,\th_w)}(u)|u\cdot n_x|du}\,.
$$
The case of an accommodation boundary condition (i.e. the so-called Maxwell-type condition considered in \cite{BGP}) corresponds to
$$
\begin{aligned}
\cK(x,v,dv'):=(1-\a(x))\de(v'-v)+\a(x)\frac{\cM_{(1,0,\th_w)}(v)|v\cdot n_x|dv'}{\displaystyle\int_{u\cdot n_x<0}\cM_{(1,0,\th_w)}(u)|u\cdot n_x|du}\,,
\end{aligned}
$$
with $0<\a(x)<1$. The case $\a\equiv 1$ corresponds to diffuse reflection, while the case $\a\equiv 0$ corresponds to specular reflection --- notice 
that specular reflection does not satisfy assumption (ii).

The class of reflection kernels described here also includes the Cercignani-Lampis model, as well as all Nocilla models. We refer the interested 
reader to section 8.4 in \cite{CIP} (on pp. 235--239) for more information on this issue.

However, this is by no means the most general class of admissible boundary conditions for the Boltzmann equation. Indeed, the boundary condition
(\ref{ScattBC}) is local in the position variable $x$. There are also models of gas-surface interaction involving boundary conditions satisfied by the
distribution function of the gas that are nonlocal is both the position and velocity variables: see for instance \cite{LomCafSam1,LomCafSam2}. This
new class of boundary condition is not covered in the present study. 

\subsection{The Incompressible Euler Scaling for the Boltzmann Equation}

The incompressible Euler limit of the Boltzmann equation is based on three different scaling prescriptions (see for instance \cite{BGL0,BGL1}).
A careful description of the dimensionless form of the Boltzmann equation can be found in section 1 of \cite{BGL2}, as well as in sections
1.9 and 1.10 of \cite{SoneBk2}.

First, as in all fluid dynamic limits of kinetic models, one assumes a strongly collisional regime. In other words, the collision integral is scaled as
$$
\cB(F,F)=\frac1{\eps^{1+q}}\cB(F,F)\,,\quad\hbox{ with }\eps\ll 1\hbox{ and }q>0\,.
$$
(See \cite{BGL2} or \cite{SoneBk2} for the physical meaning of $\eps$ and $q$.)

Next, the incompressible Euler limit holds on a long time scale, leading to the introduction of a new time variable $t=\hat t/\eps$; in other words, 
we set
$$
F(t,x,v)=\hat F_\eps(\hat t,x,v)=\hat F_\eps(\eps t,x,v)\,.
$$

With these assumptions, the scaled Boltzmann equation becomes
$$
(\eps\d_{\hat t}+v\cdot\grad_x)\hat F_\eps=\frac1{\eps^{1+q}}\hat\cB(\hat F_\eps,\hat F_\eps)\,.
$$

A third and last scaling assumption used in the incompressible Euler limit of the Boltzmann equation is that the corresponding gas flow is kept in 
a low Mach number regime. Specifically, we assume that the Mach number is of order $\Ma=\eps$, and this is done by seeking the distribution
function $\hat F_\eps$ in the form
$$
\hat F_\eps=\cM_{(1,0,\th_w)}\hat G_\eps\quad\hbox{ with }\hat G_\eps=1+\eps\hat g_\eps
$$
(where it is implicitly assumed that $g_\eps=O(1)$).

We draw the reader's attention to the fact that the parameter $\eps$ considered here has a different meaning than in the inviscid limit of the
Navier-Stokes equations. Indeed, in the situation considered here, we shall see that the reciprocal Reynolds number is $\Rey^{-1}=\eps^q$ ---
instead of $\eps$ as in the previous section.

For simplicity, we henceforth drop all hats in the scaled Boltzmann equation and consider the initial boundary value problem
\be\lb{ScalBoltz}
\left\{
\begin{array}{l}
(\eps\d_t+v\cdot\grad_x)F_\eps=\displaystyle\frac1{\eps^{1+q}}\cB(F_\eps,F_\eps)\,,\quad(x,v)\in\Om\times\bR^3\,,
\\	\\
F_\eps(t,x,v)|v\!\cdot\! n_x|=\displaystyle\int_{v'\!\cdot n_x>0}\!F_\eps(t,x,v')v'\!\cdot\! n_x\cK_\eps(x,\cR_xv,dv')\,, \quad(x,v)\in\Ga_-\,,
\\	\\
F_\eps\rstr_{t=0}=\cM_{(1,\eps u^{in},1)}\,,
\end{array}
\right.
\ee
assuming that the boundary temperature $\th_w=1$. The reflection kernel $\cK_\eps$ satisfies the properties (i)-(ii) above and possibly depends on 
the scaling parameter $\eps$. 

Henceforth we denote 
$$
M:=\cM_{(1,0,1)}\,.
$$

Let us return to the conservation laws of mass, momentum and energy satisfied by the Boltzmann collision integral. If $F_\eps$ is a smooth solution 
of (\ref{ScalBoltz}) rapidly decaying as $|v|\to+\infty$, then it satisfies the system of differential identities
\be\lb{ConsLawDiff}
\left\{
\ba
\eps\d_t\int_{\bR^3}F_\eps dv+\Div_x\int_{\bR^3}vF_\eps dv=0&\,,
\\
\eps\d_t\int_{\bR^3}vF_\eps dv+\Div_x\int_{\bR^3}v\otimes vF_\eps dv=0&\,,
\\
\eps\d_t\int_{\bR^3}\tfrac12|v|^2F_\eps dv+\Div_x\int_{\bR^3}v\tfrac12|v|^2F_\eps dv=0&\,.
\ea
\right.
\ee
These identities can be viewed again as the differential form of the local conservation laws of mass momentum and energy that are classical in
continuum mechanics.

\subsection{Main result}

Our main result in this paper is an analogue of Theorem \ref{T-NS>Eul} for the incompressible Euler limit of the Boltzmann equation.

\begin{Thm}\lb{T-Boltz>Eul}
Let $u^{in}\in C^1_c(\overline{\Om};\bR^3)$ satisfy $\Div_xu^{in}=0$ and $u\cdot n\rstr_{\d\Om}=0$, and assume that the Euler equations (\ref{Eul}) 
have a local classical solution $u\in C^1_b([0,T]\times\overline{\Om})$ defined for some $T>0$ and satisfying (\ref{uH1}). For each $\eps>0$, let 
$F_\eps$ be a solution of (\ref{ScalBoltz}). Assume that the reflection kernel $\cK_\eps$ satisfies, in addition to the properties (i)-(iii) listed above, 
\be\lb{CondK}
\left|v'_\tau dv'-\int_{v\cdot n_x>0}v_\tau\cK_\eps(x,v,dv')dv\right|\le\a_\eps(x)|v'|dv'
\ee
and
\be\lb{Cond2K}
\cK_\eps(x,v,dv')\ge\a_\eps(x)\b_\eps(x)\sqrt{2\pi}M(v)(v\cdot n_x)_+dv'
\ee
for each $x\in\d\Om$ and a.e. $v,v'\in\bR^3$ such that $v\cdot n_x>0$ and $v'\cdot n_x>0$, where 
\be\lb{Cond-a}
0\le\a_\eps(x)\le 1\quad\hbox{ and }\quad\frac1\eps\sup_{x\in\d\Om\atop|x|\le R}\a_\eps(x)\to 0\qquad\hbox{ as }\eps\to 0
\ee
and
\be\lb{Cond-b}
0<\inf_{\eps>0}\inf_{x\in\d\Om\atop|x|\le R}\b_\eps(x)\le\b_\eps(x)\le 1
\ee
for each $R>0$. Then, for each $T>0$ and $R>0$
$$
\int_0^T\int_{x\in\Om\atop |x|\le R}\left|\frac1\eps\int_{\bR^3}vF_\eps(t,x,v) dv-u(t,x)\right|dxdt\to 0\hbox{ as }\eps\to 0\,.
$$
\end{Thm}

This result was stated and proved in \cite{BGP} in the particular case where the reflection kernel $K_\eps$ corresponds with Maxwell's accomodation
condition at the boundary with accomodation parameter $\a_\eps$. 

In the case where the Boltzmann equation is set on a domain without boundary --- i.e. if $x$ runs through the Euclidian space $\bR^3$ or the periodic
box $\bT^3$ --- the first rigorous derivation of the incompressible Euler equations from the Boltzmann equation was obtained by L. Saint-Raymond
\cite{SRBoltzEul} following the relative entropy method sketched in \cite{BouGolPul,LionsMas2}. (For the simpler case of the BGK model, see also
\cite{SRBGKEul}.)

The proof of Theorem \ref{T-Boltz>Eul} is a slight generalization of the one in \cite{BGP} and is based on the same relative entropy method as in the 
work of L. Saint-Raymond \cite{SRBoltzEul}. It can also be viewed as the analogue for the Boltzmann equation of the proof of the inviscid limit Theorem
\ref{T-NS>Eul}. 

A final word of caution is in order. The statement of the theorem is left deliberately vague about the notion of solution of the Boltzmann equation 
(\ref{ScalBoltz}) to be considered. This is a rather technical matter, to be discussed later.

\subsection{The relative entropy inequality }

The first important property of the Boltzmann equation used in the proof is the following variant of Boltzmann's H theorem that can be viewed as an
analogue of the Leray energy inequality (\ref{Leray<}):
\be\lb{RelHThm}
\ba
\frac{1}{\eps^2}H(F_\eps|M)(t)&+\frac{1}{\eps^{4+q}}\int_0^t\int_\Om\cP_\eps(s,x)dxds
\\
&+\frac{1}{\eps^3}\int_0^t\int_\Ga\cD\cG_\eps(s,x)dxds\le\frac{1}{\eps^2}H(F_\eps|M)(0)
\ea
\ee
where the notation $H(F_1|F_2)$ designates the relative entropy defined as follows. 

Let $F_1\ge 0$ and $F_2>0$ a.e. on $\Om\times\bR^3$ designate two measurable functions, then
$$
\ba
H(F_1|F_2):&=\iint_{\Om\times\bR^3}\left(F_1\ln\left(\frac{F_1}{F_2}\right)-F_1+F_2\right)(x,v)dxdv
\\
&=\iint_{\Om\times\bR^3}h\left(\frac{F_1}{F_2}-1\right)F_2(x,v)dxdv\,,
\ea
$$
where $h$ is the function defined as follows:
$$
h:\,[-1,\infty)\ni z\mapsto h(z):=(1+z)\ln(1+z)-z\in\bR_+\,.
$$
(Notice that, since the integrand $h\left(\frac{F_1}{F_2}-1\right)F_2$ is a nonnegative measurable function, the relative entropy $H(F_1|F_2)$ is always
a well-defined element of $[0,+\infty]$.)

The two other quantities in (\ref{RelHThm}) are the entropy production rate per unit volume $\cP_\eps$ and the Darrozes-Guiraud information at the
boundary $\cD\cG_\eps$, whose definition is recalled below.

The entropy production rate per unit volume for the Boltzmann equation is
$$
\cP_\eps:=-\int_{\bR^3}\cB(F_\eps,F_\eps)\ln F_\eps dv
$$
and can be put in the form
$$
\cP_\eps=\iiint_{\bR^3\times\bR^3\times\bS^2}r\left(\frac{F'_\eps F'_{\eps*}}{F_\eps F_{\eps*}}-1\right)F_\eps F_{\eps*}|(v-v_*)\cdot\om|dvdv_*d\om\,,
$$
where $r$ is the following function:
$$
r:\,(-1,\infty)\ni z\mapsto r(z):=z\ln(1+z)\in\bR_+\,.
$$

As for the Darrozes-Guiraud information, it is defined as 
$$
\cD\cG_\eps(t,x):=\int_{\bR^3}h\left(\frac{F_\eps}{M}-1\right)(t,x,v)v\cdot n_xMdv\,,\quad x\in\d\Om\,,\,\,t>0\,.
$$
Observe that, for each $x\in\d\Om$ and each $v\in\bR^3$ such that $v\cdot n_x>0$, 
$$
\mu^\eps_{x,v}(dv'):=\frac{(v'\cdot n_x)_+M(v')\cK_\eps(x,v,dv')}{(v\cdot n_x)_+M(v)}
$$
is a probability measure on $\bR^3$ (by property (iii) of $\cK_\eps$), and that
$$
\frac{F_\eps}{M}(t,x,\cR_xv)=\La\mu^\eps_{x,v},\frac{F_\eps}{M}\Ra(t,x,v)\,,\quad\hbox{ for each }(x,v)\in\Ga_+\hbox{ and }t\ge 0\,.
$$
Since $h$ is convex, it follows from Jensen's inequality that
$$
\La\mu^\eps_{x,v},h\left(\frac{F_\eps}{M}-1\right)\Ra\ge h\left(\La\mu^\eps_{x,v},\frac{F_\eps}{M}-1\Ra\right)\,.
$$
On the other hand, by property (i) of the reflection kernel $\cK_\eps$, one has
$$
\int\La\mu^\eps_{x,v},h\left(\frac{F_\eps}{M}-1\right)\Ra M(v\cdot n_x)_+dv
=
\int h\left(\frac{F_\eps}{M}-1\right)M(v\cdot n_x)_+dv\,.
$$
Hence the Darrozes-Guiraud information satisfies
$$
\cD\cG_\eps(t,x)=\int_{\bR^3}\left(h\left(\frac{F_\eps}{M}-1\right)-h\left(\La\mu^\eps_{x,v},\frac{F_\eps}{M}-1\Ra\right)\right)M(v\cdot n_x)_+dv\ge 0\,.
$$
(See for instance Theorem 8.5.1 on p. 240 in \cite{CIP}.)

As a consequence of the local conservation of mass, i.e the first equality in (\ref{ConsLawDiff}) one finds that, for each scalar test function
$\phi\in C^1_c(\bR_+\times\overline\Om)$ and each $t>0$,
\be\lb{LocConsMassDistrib}
\ba
\eps\iint_{\Om\times\bR^3}F_\eps(t,x,v)\phi(t,x)dxdv-\eps\iint_{\Om\times\bR^3}F_\eps(0,x,v)\phi(0,x)dxdv&
\\
+\int_0^t\iint_{\d\Om\times\bR^3}F_\eps(s,x,v)\phi(s,x)v\cdot n_xdS(x)dvds&
\\
=\int_0^t\iint_{\Om\times\bR^3}F_\eps(s,x,v) (\eps\d_t+v\cdot\grad_x)\phi(s,x) dsdxdv&\,.
\ea
\ee

Let now $w\in C^1_c([0,T]\times\overline{\Om};\bR^3)$ satisfy 
\be\lb{Cond-w}
\Div_xw=0\,,\quad\hbox{ and }w\cdot n\rstr_{\d\Om}\,.
\ee
As a consequence of the local conservation of momentum, i.e. the second identity in (\ref{ConsLawDiff}), one finds that, for each $t>0$,
\be\lb{LocConsMomDistrib}
\ba
\eps\iint_{\Om\times\bR^3}F_\eps(t,x,v)v\cdot w(t,x)dxdv-\eps\iint_{\Om\times\bR^3}F_\eps(0,x,v)v\cdot w(0,x)dxdv&
\\
+\int_0^t\iint_{\d\Om\times\bR^3}F_\eps(s,x,v)(v_\tau\cdot w_\tau)v\cdot n_xdS(x)dvds&
\\
=\int_0^t\iint_{\Om\times\bR^3}F_\eps(s,x,v) (\eps\d_t+v\cdot\grad_x)(v\cdot w)(s,x) dsdxdv&\,.
\ea
\ee

Finally, observe that
\be\lb{RelEnt-w}
\ba
H(F_\eps|\cM_{(1,\eps w,1)})&=H(F_\eps|M)+\iint_{\Om\times\bR^3}F_\eps\ln\left(\frac{M}{\cM_{(1,\eps w,1)}}\right)dxdv
\\
&=H(F_\eps|M)+\iint_{\Om\times\bR^3}F_\eps\left(\tfrac12\eps^2|w|^2-\eps w\cdot v\right)dxdv\,.
\ea
\ee

Putting together (\ref{RelEnt-w}), (\ref{RelHThm}),  (\ref{LocConsMassDistrib}) and (\ref{LocConsMomDistrib}), we find that, for each test velocity 
field  $w\in C^1_c([0,T]\times\overline{\Om};\bR^3)$ satisfying (\ref{Cond-w}) and each $t>0$
\be\lb{RelEntr<}
\ba
\frac1{\eps^2}H&(F_\eps|\cM_{(1,\eps w,1)})(t)\le\frac1{\eps^2}H(F^{in}_\eps|\cM_{(1,\eps w(0,\cdot),1)})
\\
&-\frac1{\eps^{4+q}}\int_0^t\int_\Om\cP_\eps(s,x)dxds-\frac{1}{\eps^3}\int_0^t\int_{\d\Om}\cD\cG_\eps(s,x)dS(x)ds
\\
&-\frac1{\eps^2}\int_0^t\iint_{\Om\times\bR^N}(v-\eps w(s,x))^{\otimes 2}:\grad_xw(s,x)F_\eps(s,x,v)dxdvds
\\
&-\frac1\eps\int_0^t\iint_{\Om\times\bR^N}(v-\eps w(s,x))\cdot E(w)(s,x)F_\eps(s,x,v)dvdxds
\\
&+\frac{1}{\eps^2}\int_0^t\iint_{\d\Om\times\bR^N}F_\eps(s,x,v)(v_\tau\cdot w_\tau)(s,x)(v\cdot n_x)dS(x)dvds\,,
\ea
\ee
where $E(w)$ has been defined in (\ref{DefE}).

This inequality is the analogue for the scaled Boltzmann equation of the inequality (\ref{DiffEner}) used in the proof of Theorem \ref{T-NS>Eul}.
Indeed, for each pair $u_1, u_2$ of (measurable) vector fields in $\Om$ such that
$$
\int_{\Om}|u_1(x)|^2dx+\int_{\Om}|u_2(x)|^2dx<+\infty\,,
$$
one has
$$
\frac1{\eps^2}H(\cM_{(1,\eps u_2,1)}|\cM_{(1,\eps u_1,1)})\to\tfrac12\int_{\Om}|u_1-u_2|^2(x)dx\,.
$$
Therefore, the scaled relative entropy
$$
\frac1{\eps^2}H(F_\eps|\cM_{(1,\eps w,1)})\hbox{ is the analogue of }\tfrac12\int_\Om|u_\eps-u|^2(t,x)dx\,,
$$
the entropy production
$$
\frac1{\eps^{4+q}}\int_0^t\int_\Om\cP_\eps(s,x)dxds
$$
is the analogue of the viscous dissipation
$$
\eps\int_0^t\int_\Om\tfrac12|\Si(u_\eps)(s,x)|^2dxds\,,
$$
--- with the notation and scaling assumption used in sections \ref{S-NS>E}-\ref{S-NS2>E} --- while the Darrozes-Guiraud information 
$$
\frac{1}{\eps^3}\int_0^t\int_{\d\Om}\cD\cG_\eps(s,x)dS(x)ds
$$
is the analogue of the boundary friction
$$
\l_\eps\int_0^t\int_{\d\Om}|u_\eps(s,x)|^2dS(x)ds
$$
appearing in section \ref{S-NS2>E}. Likewise the term
$$
\frac1{\eps^2}\int_0^t\iint_{\Om\times\bR^N}(v-\eps w(s,x))^{\otimes 2}:\grad_xw(s,x)F_\eps(s,x,v)dxdvds
$$
is the analogue of 
$$
\int_0^t\int_\Om\grad_xw:(u_\eps-w)^{\otimes 2}(s,x)dxds\,,
$$
and we seek to control it in terms of the scaled relative entropy
$$
\frac1{\eps^2}H(F_\eps|\cM_{(1,\eps w,1)})
$$
in order to conclude with Gronwall's inequality. This is precisely what is done in \cite{SRBoltzEul}, and we shall not repeat this (difficult) analysis
here.

What remains to be done is to control the boundary term
$$
\frac{1}{\eps^2}\int_0^t\iint_{\d\Om\times\bR^N}F_\eps(s,x,v)(v_\tau\cdot w_\tau)(s,x)(v\cdot n_x)dS(x)dvds
$$
of indefinite sign by the Darrozes-Guiraud inequality, possibly up to some asymptotically negligible quantity. This step can be viewed as the
analogue for the Boltzmann equation of the Cauchy-Schwarz inequality used to control the boundary term
$$
\l_\eps\int_0^t\int_{\d\Om}u_\eps\cdot u(s,x)dS(x)ds
$$
by the boundary friction
$$
\l_\eps\int_0^t\int_{\d\Om}|u_\eps|^2(s,x)dS(x)ds\,.
$$

\subsection{Proof of Theorem \ref{T-Boltz>Eul}}

The key argument in the proof of Theorem \ref{T-Boltz>Eul} is given by the following proposition.

\begin{Prop}\lb{P-BoundCtrol}
Under the same assumptions as in Theorem \ref{T-Boltz>Eul}, for each vector field $w\in C^1_c(\bR_+\times\overline\Om;\bR^3)$ such that 
$\Div_xw=0$ and $w\cdot n\rstr_{\d\Om}=0$, satisfying $\Supp(w)\subset\bR_+\times K$, where $K$ is a compact subset of $\bR^3$, one has
$$
\ba
{}&\left|\int_0^t\iint_{\d\Om\times\bR^3}F_\eps(s,x,v)(v_\tau\cdot w_\tau)(s,x)(v\cdot n_x)dS(x)dvds\right|
\\
&\qquad\qquad\qquad\le
\left(\frac1{N\eps}+\frac{\sqrt{2\pi}C(w,N)}{h(1/2)}\eps\right)\int_0^t\int_{\d\Om\cap K}\frac1{\b_\eps(x)}\cD\cG_\eps(s,x)dxds
\\
&\qquad\qquad\qquad+
2C(w,N)\eps^2\int_0^t\int_{\d\Om\cap K}\frac{\a_\eps(x)}{\eps}\int_{\bR^3}F_\eps(s,x,v)(v\cdot n_x)^2dvdxds
\ea
$$
for each $\eps\in (0,1)$, each $N\ge 1$ and each $t\ge 0$, where
\be\lb{DefCwN}
C(w,N):=\tfrac1N\int_{\bR^3}(e^{N\|w\|_{L^\infty}|v|}-N\|w\|_{L^\infty}|v|-1)(v\cdot n_x)_+Mdv\,.
\ee
\end{Prop}

Set
\be\lb{DefGg}
F_\eps=MG_\eps\,,\quad G_\eps=1+\eps g_\eps\,.
\ee
Using the boundary condition in (\ref{ScalBoltz}) and the substitution $v\mapsto\cR_xv$ for $v\cdot n_x<0$ shows that
$$
\ba
\int_0^t\iint_{\d\Om\times\bR^3}F_\eps(s,x,v)(v_\tau\cdot w_\tau)(s,x)(v\cdot n_x)dS(x)dvds
\\
=
\int_0^t\iint_{\d\Om\times\bR^3}(F_\eps(s,x,v)-F_\eps(s,x,\cR_xv)(v_\tau\cdot w_\tau)(s,x)(v\cdot n_x)_+dS(x)dvds
\ea
$$
--- notice that $\cR_xv_\tau=v_\tau$. Thus
$$
\ba
\int_0^t\iint_{\d\Om\times\bR^3}F_\eps(s,x,v)(v_\tau\cdot w_\tau)(s,x)(v\cdot n_x)dS(x)dvds
\\
=
\eps\int_0^t\iint_{\d\Om\times\bR^3}(g_\eps(s,x,v)-g_\eps(s,x,\cR_xv)(v_\tau\cdot w_\tau)(s,x)(v\cdot n_x)_+M(v)dS(x)dvds
\\
=
\eps\int_0^t\iint_{\d\Om\times\bR^3}(g_\eps-\la\mu^\eps_{x,v},g_\eps\ra)(s,x,v)(v_\tau\cdot w_\tau)(s,x)(v\cdot n_x)_+M(v)dS(x)dvds\,.
\ea
$$

Since $\mu^\eps_{x,v}$ is a probability measure acting on the velocity variable only, 
$$
g_\eps-\la\mu^\eps_{x,v},g_\eps\ra=g_\eps-\L_xg_\eps-\la\mu^\eps_{x,v},g_\eps-\L_xg_\eps\ra\,,
$$
where
\be\lb{DefLambda}
\L_x\phi:=\sqrt{2\pi}\int_{\bR^3}\phi(v)(v\cdot n_x)_+M(v)dv\,.
\ee

Therefore
$$
\ba
\left|\int_0^t\iint_{\Ga}F_\eps(s,x,v)(v_\tau\cdot w_\tau(s,x,v))(v\cdot n_x)dS(x)dvds\right|
\\
=\eps\left|\int_0^t\iint_{\Ga_+}(g_\eps-\L_xg_\eps-\la\mu^\eps_{x,v},g_\eps-\L_xg_\eps\ra)(v_\tau\cdot w_\tau)(v\cdot n_x)_+M(v)dv\right|
\\
\le\eps\int_0^t\iint_{\Ga_+}\a_\eps(x)|g_\eps-\L_xg_\eps|(s,x,v)|v_\tau||w_\tau|(s,x)(v\cdot n_x)_+M(v)dv 
\ea
$$
by assumption (\ref{CondK}).

At this point, the proof of Proposition \ref{P-BoundCtrol} is done in two steps.  

The first step is summarized in the following lemma.

\begin{Lem}\lb{L-Young<}
Under the same assumptions as in Proposition \ref{P-BoundCtrol} and with the notation (\ref{DefGg}), for each $(t,x)\in(\bR_+\times\d\Om)\cap\Supp(w)$, 
each $N\ge 1$ and each $\eps\in(0,1)$, 
$$
\ba
N\eps^2\int_{\bR^3}|g_\eps-\L_xg_\eps||v_\tau||w_\tau|(v\cdot n_x)_+M(v)dv&
\\
\le
\int_{\bR^3}(h(\eps g_\eps)-h(\eps\L_x(g_\eps)))(v\cdot n_x)_+M(v)dv+NC(w,N)\eps^2\L_x(F_\eps)&\,,
\ea
$$
where $\L_x$ has been defined in (\ref{DefLambda}) and $C(w,N)$ in (\ref{DefCwN}).
\end{Lem}

\begin{proof}
The proof follows that of Lemma 4.6 in \cite{BGP}. Pick $z_0>-1$, and define, for each $z>-1$
$$
l(z-z_0):=h(z)-h(z_0)-h'(z_0)(z-z_0)\,.
$$
We recall that the Legendre dual of $h$ and $l$ are
$$
h^*(p)=e^p-p-1\quad\hbox{ and }l^*(p)=(1+z_0)h^*(p)=(1+z_0)(e^p-p-1)\,.
$$

Applying Young's inequality (see \cite{BGL2}, especially section 3 there) to $l$ with $z=\eps g_\eps$ and $z_0=\eps\L_x(g_\eps)$
$$
\ba
N\eps^2|g_\eps-\L_x(g_\eps)||v_\tau||w_\tau|
\le
l(\eps(g_\eps-\L_x(g_\eps)))+l^*(N\eps\Sign(g_\eps-\L_x(g_\eps))|v_\tau||w_\tau|)&
\\
=
l(\eps(g_\eps-\L_x(g_\eps)))+(1+\eps\L_x(g_\eps))(e^{N\eps|w||v|}-N\eps|w||v|-1)&
\\
\le l(\eps(g_\eps-\L_x(g_\eps)))+\eps^2(1+\eps\L_x(g_\eps))(e^{N|w||v|}-N|w||v|-1)&\,,
\ea
$$
whenever $0<\eps<1$. (Indeed, for each $z>0$ and $\eps\in(0,1)$, one has
$$
\ba
e^{\eps z}-\eps z-1&=\sum_{n\ge 2}\frac{(\eps z)^n}{n!}
\\
&\le\eps^2\sum_{n\ge 2}\frac{z^n}{n!}=\eps^2(e^z-z-1)\,;
\ea
$$
see also (3.14) in \cite{BGL2}.)

Therefore
\be\lb{Interm<}
\ba
N\eps^2\int_{\bR^3}|g_\eps-\L_x(g_\eps)|(s,x,v)|v_\tau||w_\tau|(s,x)(v\cdot n_x)_+M(v)dv&
\\
\le
\int_{\bR^3}l(\eps(g_\eps-\L_x(g_\eps)))(v\cdot n_x)_+M(v)dv+NC(w,N)\eps^2\L_x(F_\eps)&\,.
\ea
\ee

Moreover, since $\L_x$ is constant in $v$ and the average under a probability measure,
$$
\ba
\int_{\bR^3}h'(\eps\L_x(g_\eps))&(g_\eps-\L_x(g_\eps))(v\cdot n_x)_+Mdv
\\
&=h'(\eps\L_x(g_\eps))\int_{\bR^3}(g_\eps-\L_x(g_\eps))(v\cdot n_x)_+Mdv
\\
&=h'(\eps\L_x(g_\eps))\tfrac1{\sqrt{2\pi}}\L_x(g_\eps-\L_x(g_\eps))=0\,,
\ea
$$
so that
$$
\ba
\int_{\bR^3}&l(\eps(g_\eps-\L_x(g_\eps)))(v\cdot n_x)_+M(v)dv
\\
&=
\int_{\bR^3}(h(\eps g_\eps)-h(\eps\L_x(g_\eps)))(v\cdot n_x)_+M(v)dv\,.
\ea
$$
Substituting this last integral in the right hand side of (\ref{Interm<}) leads to the inequality in the statement of Lemma \ref{L-Young<}.
\end{proof}

It remains to relate the integral on the right hand side of the inequality in Lemma \ref{L-Young<} to the Darrozes-Guiraud information.

\begin{Lem}\lb{L-hL<DG}
Under the same assumptions as in Proposition \ref{P-BoundCtrol} and with the notation (\ref{DefGg}), for each $(t,x)\in\bR_+\times\d\Om$ and each 
$\eps\in(0,1)$, one has
$$
\a_\eps(x)\b_\eps(x)\int_{\bR^3}(h(\eps g_\eps)-h(\eps\L_x(g_\eps)))(t,x,v)(v\cdot n_x)_+M(v)dv\le\cD\cG_\eps(t,x)\,.
$$
\end{Lem}

\begin{proof}
By (\ref{Cond2K}), the reflection kernel $\cK$ can be put in the form
$$
\cK(x,v,dv')=\a_\eps(x)\b_\eps(x)\sqrt{2\pi}(v\cdot n_x)_+M(v)dv+(1-\a_\eps(x)\b_\eps(x))\cL_\eps(x,v,dv')\,.
$$
Since $0\le\a_\eps(x)\b_\eps(x)\le 1$, the kernel $\cL_\eps(x,v,dv')$ so defined satisfies the properties (i), (ii) and 

\smallskip
\noindent
(iii') for each $x\in\d\Om$, 
$$
M(v)v\!\cdot\! n_x=\int_{v'\!\cdot n_x>0}M(v')v'\!\cdot\! n_x\cL_\eps(x,v,dv')\,,\quad v\!\cdot\! n_x>0\,.
$$

Define $\nu_{x,v}$ by analogy with $\mu_{x,v}$:
$$
\nu^\eps_{x,v}(dv'):=\frac{(v'\cdot n_x)_+M(v')\cL_\eps(x,v,dv')}{(v\cdot n_x)_+M(v)}\,,
$$
so that the boundary condition satisfied by the solution $F_\eps$ of (\ref{ScalBoltz}) takes the form
$$
\ba
\frac{F_\eps}{M}(t,x,\cR_xv)=\a_\eps(x)\b_\eps(x)\L_x\left(\frac{F_\eps}{M}\right)+(1-\a_\eps(x)\b_\eps(x))\La\nu^\eps_{x,v},\frac{F_\eps}{M}\Ra(t,x,v)\,,
\\
\quad\hbox{ for each }(x,v)\in\Ga_+\hbox{ and }t\ge 0\,,
\ea
$$
or equivalently
$$
\ba
g_\eps(t,x,\cR_xv)=\a_\eps(x)\b_\eps(x)\L_x(g_\eps)(t,x,v)+(1-\a_\eps(x)\b_\eps(x))\la\nu^\eps_{x,v},g_\eps(t,x,\cdot)\ra\,,
\\
\quad\hbox{ for each }(x,v)\in\Ga_+\hbox{ and }t\ge 0\,.
\ea
$$

Thus
$$
\ba
\cD\cG_\eps(t,x)&=\int h(\eps g_\eps)(t,x,v)v\cdot n_xM(v)dv
\\
&=\int (h(\eps g_\eps)(t,x,v)-h(\eps g_\eps)(t,x,\cR_xv))(v\cdot n_x)_+M(v)dv
\ea
$$
which can be put in the form
$$
\cD\cG_\eps=\tfrac1{\sqrt{2\pi}}\L_x(h(\eps g_\eps)-h(\a_\eps\b_\eps(x)\L_x(g_\eps)+(1-\a_\eps\b_\eps(x))\la\nu^\eps_{x,v},g_\eps\ra))
$$

By convexity of $h$
$$
\ba
{}&h(\a_\eps\b_\eps(x)\L_x(g_\eps)+(1-\a_\eps\b_\eps(x))\la\nu^\eps_{x,v},g_\eps\ra)
\\
&\qquad\le\a_\eps\b_\eps(x)h(\L_x(g_\eps))+(1-\a_\eps\b_\eps(x))h(\la\nu^\eps_{x,v},g_\eps\ra)
\ea
$$
so that
$$
\ba
\cD\cG_\eps&\ge\tfrac1{\sqrt{2\pi}}\a_\eps\b_\eps(x)\L_x(h(\eps g_\eps)-h(\L_x(g_\eps)))
\\
&+\tfrac1{\sqrt{2\pi}}(1-\a_\eps\b_\eps(x))\L_x(h(\eps g_\eps)-h(\la\nu^\eps_{x,v},g_\eps\ra))
\ea
$$
Since the kernel $\cL_\eps$ satisfies property (i), one has
$$
\L_x(h(\eps g_\eps))=\L_x(\la\nu^\eps_{x,v},h(\eps g_\eps)\ra)
$$
so that
$$
\L_x(h(\eps g_\eps)-h(\la\nu^\eps_{x,v},g_\eps\ra))=\L_x(\la\nu^\eps_{x,v},h(\eps g_\eps)-h(\la\nu^\eps_{x,v},g_\eps\ra)\ra)\ge 0\,,
$$
where the last inequality follows from Jensen's inequality applied to the probability measure $\nu^\eps_{x,v}$ and the convex function $h$.

Since $0\le\a_\eps(x)\b_\eps(x)\le 1$ for all $x\in\d\Om$, we conclude that 
$$
\cD\cG_\eps\ge\tfrac1{\sqrt{2\pi}}\a_\eps\b_\eps(x)\L_x(h(\eps g_\eps)-h(\L_x(g_\eps)))
$$
which is precisely the sought inequality.
\end{proof}

Finally, we control the outgoing mass flux at the boundary exactly as explained in Lemma 4.7 of \cite{BGP}.

\begin{Lem}\lb{L-OutFlux}
Under the same assumptions and with the same notation as in Proposition \ref{P-BoundCtrol}, for each $(t,x)\in\bR_+\times\d\Om$ and each 
$\eps,\eta\in(0,1)$, one has
$$
\ba
\a_\eps(x)\int_{\bR^3}F_\eps(t,x,v)(v\cdot n_x)_+M(v)dv&\le\frac1{h(\eta)\b_\eps(x)}\cD\cG_\eps(t,x)
\\
&+\frac{\a_\eps(x)}{\sqrt{2\pi}(1-\eta)}\int_{\bR^3}F_\eps(t,x,v)(v\cdot n_x)^2dv\,.
\ea
$$
\end{Lem}

\begin{proof}
Following the argument in Lemma 4.7 of \cite{BGP} shows that
$$
\ba
\int_{\bR^3}F_\eps(t,x,v)(v\cdot n_x)_+M(v)dv&
\\
\le\frac1{h(\eta)}\int_{\bR^3}\left(G_\eps\ln\left(\frac{G_\eps}{\L_x(G_\eps)}\right)-G_\eps+\L_x(G_\eps)\right)(t,x,v)(v\cdot n_x)_+M(v)dv&
\\
+\frac1{\sqrt{2\pi}(1-\eta)}\int_{\bR^3}F_\eps(t,x,v)(v\cdot n_x)^2dv&
\\
=\frac1{h(\eta)}\int_{\bR^3}(h(\eps g_\eps)-h(\eps\L_x(g_\eps)))(t,x,v)(v\cdot n_x)_+M(v)dv&
\\
+\frac1{\sqrt{2\pi}(1-\eta)}\int_{\bR^3}F_\eps(t,x,v)(v\cdot n_x)^2dv&\,.
\ea
$$
We conclude by applying Lemma \ref{L-hL<DG} to the first integral on the right hand side of the last equality above.
\end{proof}

Putting together the inequalities in Lemmas \ref{L-Young<}, \ref{L-hL<DG} and \ref{L-OutFlux} with $\eta=\tfrac12$ leads to the estimate in 
Proposition \ref{P-BoundCtrol}.

Next we return to (\ref{RelEntr<}), and observe that the boundary term can be absorbed by the Darrozes-Guiraud information, as follows. 

Pick a compactly supported vector field $w$ of class $C^1$ such that $\Div_xw=0$, $w\cdot n\rstr_{\d\Om}=0$ and let $K$ be a compact subset 
of $\bR^3$ such that $\Supp(w)\subset\bR_+\times K$. Set
$$
r_\eps:=\sup_{x\in K\cap\d\Om}\frac{\a_\eps(x)}{\eps}\quad\hbox{ and }\quad\b_*:=\inf_{0<\eps<1\atop x\in K\cap\d\Om}\b_\eps(x)>0\,;
$$
--- see assumptions (\ref{Cond-a})-(\ref{Cond-b}). Pick $N\ge 1$ such that $N\b_*>1$. It follows from Proposition \ref{P-BoundCtrol} that
\be\lb{RelEntr<Fin}
\ba
\frac1{\eps^2}H(F_\eps|\cM_{1,\eps w,1})(t)\le\frac1{\eps^2}H(F^{in}_\eps|\cM_{1,\eps w(0,\cdot),1})
	-\frac1{\eps^{4+q}}\int_0^t\int_\Om\cP_\eps(s,x)dxds&
\\
-\frac{1}{\eps^3}\left(1-\frac1{N\b_*}-\frac{\sqrt{2\pi}C(w,N)}{h(1/2)\b_*}\eps^2\right)\int_0^t\int_{\d\Om}\cD\cG_\eps(s,x)dS(x)ds&
\\
-\frac1{\eps^2}\int_0^t\iint_{\Om\times\bR^3}(v-\eps w(s,x))^{\otimes 2}:\grad_xw(s,x)F_\eps(s,x,v)dxdvds&
\\
-\frac1\eps\int_0^t\iint_{\Om\times\bR^3}(v-\eps w(s,x))\cdot E(w)(s,x)F_\eps(s,x,v)dvdxds&
\\
+2C(w,N)r_\eps\int_0^t\iint_{(\d\Om\cap K)\times\bR^3}F_\eps(s,x,v)(v\cdot n_x)^2dS(x)dvds&\,.
\ea
\ee

At this point, we recall that solutions of the initial boundary value problem (\ref{ScalBoltz})  Boltzmann equation satisfy the bound
$$
\sup_{\eps>0}\int_0^t\iint_{(\d\Om\cap K)\times\bR^3}F_\eps(s,x,v)(v\cdot n_x)^2dS(x)dvds<\infty
$$
for each $t>0$ and each compact $K\subset\bR^3$ --- see \cite{MasmLSR} and the proof of Theorem 4.1 d) in \cite{BGP}. 

Since $r_\eps\to 0$ as $\eps\to 0$ by our assumption (\ref{CondK})-(\ref{Cond-a}) on the reflection kernel $\cK_\eps$, and 
$$
\left(1-\frac1{N\b_*}-\frac{\sqrt{2\pi}C(w,N)}{h(1/2)\b_*}\eps^2\right)>0
$$
for all small enough $\eps>0$, the inequality (\ref{RelEntr<Fin}) is of the same form as the inequality stated as Theorem 5 in \cite{SRBoltzEul}. One 
then concludes by the same argument as in \cite{SRBoltzEul}.


\section{Conclusion and Final Remarks}


In this work, we have derived the incompressible Euler equations from the Boltzmann equation in the case of the initial boundary value problem, 
for a class of boundary conditions at the kinetic level of description that is more general than the case of Maxwell accomodation considered in
\cite{BGP}. However, the asymptotic regime for these boundary conditions considered here is the same as the one considered in \cite{BGP}: 
condition (\ref{CondK})-(\ref{Cond-a}) means that, in the vanishing $\eps$ limit, the tangential momentum of gas molecules tends to be conserved. 
In the case of Maxwell's accomodation condition at the boundary, this is the case precisely when the accomodation parameter vanishes, so that 
Maxwell's condition approaches specular reflection. The (unrealistic) specular reflection condition is the analogue for the Boltzmann equation 
of the (equally unrealistic) Navier full slip boundary condition, for which the inviscid limit of the Navier-Stokes equation is known to be described 
by the Euler equation. In that sense, Theorem \ref{T-Boltz>Eul} is the analogue for the Boltzmann equation of the (much simpler) Theorem
\ref{T-NS>Eul} already established in \cite{BGP}.

There is no single notion of accomodation parameter for boundary conditions defined by a reflection kernel as in (\ref{ScattBC}), as in the case
of Maxwell's condition. There is however a notion of accomodation parameter that can be defined for each mechanical quantity attached to
a gas molecule impinging on the material surface: see the discussion in section 8.3 in \cite{CIP}, and especially formulas (3.4) and (3.11) there.
The interested reader is invited to compare that definition of accomodation parameter with the condition (\ref{CondK})-(\ref{Cond-a}) used in 
Theorem \ref{T-Boltz>Eul}.

We have already stressed the striking analogy between our proofs of Theorems \ref{T-NS>Eul} and \ref{T-Boltz>Eul}. In order to go further in 
the analysis of the incompressible Euler limit of the Boltzmann equation in the presence of material boundaries, it seems natural to investigate
the similarities between the various terms that appear in the energy inequality (\ref{DiffEner}) in the Navier-Stokes case, and the relative 
entropy inequality (\ref{RelEntr<}) in the Boltzmann case. The entropy production rate $\eps^{-4-q}\cP_\eps$ in (\ref{RelEntr<}) is clearly the
analogue of the viscous dissipation rate $\eps|\Si(u_\eps)|^2$ (with the notation and scaling assumptions used in sections \ref{S-NS>E}
and \ref{S-NS2>E}: see Proposition 4.6 in \cite{BGL2}, and especially formula (4.18) there\footnote{This formula is established under the 
assumption of the Navier-Sokes scaling corresponding with $q=0$ in the notation of the present paper; extending it to the Euler scaling is 
straightforward. The reader is also invited to pay attention to the definition of $q$ in that reference, whose meaning is different than in the 
present paper.}. The analogy with the Darrozes-Guiraud information $\eps^{-3}\cD\cG_\eps$ is already somewhat less clear. The boundary 
friction term $\l_\eps|u_\eps|^2$ appearing in (\ref{DiffEner}) is partly analogous to the Darrozes-Guiraud --- but only partly so, since the 
Darrozes-Guiraud information also controls the departure of the distribution function of gas molecules impinging on the material boundary 
from thermodynamic equilibrium. For instance, the Darrozes-Guiraud information term is present except in the  case of specular reflection of 
gas molecules at the boundary, whereas the boundary friction term vanishes identically if the solutions of the Navier-Stokes equations are 
assumed to satisfy the Dirichlet boundary condition.

If the family of solutions $u_\eps$ of the Navier-Stokes equations (\ref{NSDir}) or (\ref{NSSlip}) converges to the solution of the Euler equations
(\ref{Eul}) as $\eps\to 0$, then
$$
\eps|\Si(u_\eps)|^2\to 0\hbox{ in }L^2((0,T)\times\Om)
$$
where $T$ is the life-time of the Euler solution, and, in the case where the Navier-Stokes solutions satisfy a slip boundary condition with slip
coefficient $\l_\eps$,
$$
\l_\eps|u_\eps\rstr_{\d\Om}|^2\to 0\hbox{ in }L^2((0,T)\times\d\Om)
$$
as $\eps\to 0$. (Indeed, the kinetic energy of the fluid is an invariant of the motion under the Euler dynamics.) Likewise, if the family $F_\eps$ of
solutions of (\ref{ScalBoltz}) satisfies
$$
\frac1\eps\int_{\bR^3}vF_\eps dv\to u\hbox{ and }\frac1{\eps^2}H(F_\eps|\cM_{(1,\eps u,1)})\to 0
$$
as $\eps\to 0$, where $u$ is the solution of (\ref{Eul}), then
$$
\frac1{\eps^{4+q}}\cP_\eps\to 0\hbox{ in }L^1((0,T)\times\Om)\hbox{ and }\frac1{\eps^3}\cD\cG_\eps\to 0\hbox{ in }L^1((0,T)\times\d\Om)\,.
$$
In the case where the Navier-Stokes solutions satisfy the Dirichlet condition, it is easily seen that the Navier-Stokes solutions converge to 
the Euler solution if and only if $\eps(\Si(u_\eps)\cdot n)_\tau\to 0$ in $\cD'((0,T)\times\d\Om))$. In Kato's paper \cite{KatoBL}, this term is 
controlled by only the viscous dissipation term, by a very clever argument involving the construction of a boundary layer different from the
Prandtl construction together with Hardy's inequality. Whether there is an analogue of this argument for the Boltzmann equation remains
an open question at the time of this writing.

Another issue related to the ones considered here is the incompressible Navier-Stokes limit of the Boltzmann equation with boundary
condition of the type (\ref{ScattBC}). At the formal level, this question is investigated in detail in \cite{SoneBk2}. However, the discussion in
\cite{SoneBk2} leaves aside the special case where the Navier-Stokes limit leads to a slip boundary condition. In the case of the BGK 
model with Maxwell's accomodation condition at the boundary, this problem has been treated (at the formal level) by Aoki-Inamuro-Onishi 
in \cite{Aoki}, by the Hilbert expansion method, completed with appropriate Knudsen layer terms. In the case of the Boltzmann equation, 
under a scaling assumption leading to the evolution Stokes equations, the slip boundary condition has been obtained by Masmoudi and 
Saint-Raymond \cite{MasmLSR} by a rigorous argument involving the same moment method as in \cite{BGL1,GL} together with the weak
(i.e. in the sense of distributions) formulation of the slip boundary condition. The Navier-Stokes analogue of this result has been obtained
(at the formal level) in section 3 of \cite{BGP}. 

While both approaches lead to the same value of the slip coefficient, a natural question is to study the Navier-Stokes limit of the Boltzmann 
equation with the most general boundary condition of the form (\ref{ScattBC}), and to identify under which condition(s) on the reflection 
kernel $\cK$, the limiting velocity field satisfies the slip boundary condition. In particular, one should check that the Aoki-Inamuro-Onishi 
theory and the weak formulation of the momentum equations lead to the same boundary condition in the vanishing $\eps$ limit. We hope 
to return to these questions in a future publication \cite{Golse}.

Finally, we conclude with remarks of a more technical nature, on the notion of solution of the initial boundary value problem for the Boltzmann
equation. A careful examination of the proof of Theorem \ref{T-Boltz>Eul} reveals that, in addition to the properties of solutions already used
in \cite{SRBoltzEul} in the absence of material boundaries, all that is needed is the relative variant of Boltzmann's H theorem (\ref{RelHThm}),
the weak form of the local conservation of mass (continuity equation) (\ref{LocConsMassDistrib}) and the weak form of the local conservation
of momentum (\ref{LocConsMomDistrib}). Recently, Mischler \cite{Mischler} constructed global, renormalized (in the sense of DiPerna-Lions
\cite{dPL}) solutions of the Boltzmann equation in (\ref{ScalBoltz}), satisfying (\ref{RelHThm}) and (\ref{LocConsMassDistrib}). However, the
boundary condition in (\ref{ScalBoltz}) should be replaced with the inequality
$$
F_\eps(t,x,v)|v\!\cdot\! n_x|\le\displaystyle\int_{v'\!\cdot n_x>0}\!F_\eps(t,x,v')v'\!\cdot\! n_x\cK_\eps(x,\cR_xv,dv')\,, \quad(x,v)\in\Ga_-\,.
$$ 
Equality in the boundary condition is known at least in the case of Maxwell's accomodation condition --- see Remark 6.4 in \cite{Mischler}.
The reason for these differences with what would be normally expected from classical solutions of (\ref{ScalBoltz}) is that Mischler's solutions 
are obtained by a very delicate compactness procedure in some weak topology. Whether these solutions are uniquely determined by their
initial data, or even satisfy the weak form of the local momentum conservation (\ref{LocConsMomDistrib}) remains yet unknown. On the other
hand, obtaining classical solutions of (\ref{ScalBoltz}) for all $\eps>0$ and for all compactly supported and solenoidal initial velocity fields 
$u^{in}$ of class $C^1$ in the domain $\Om$ and tangential on $\d\Om$ remains an open problem at the time of this writing. Since the only 
properties of solutions used in our proof of the Euler limit are (\ref{RelHThm}), (\ref{LocConsMassDistrib}) and (\ref{LocConsMomDistrib}), 
one could hope that some future refinement of Mischler's theory could be enough for our purposes. 

On the other hand, the Euler limit stated in Theorem \ref{T-Boltz>Eul} can be formulated in terms of dissipative solutions of the incompressible
Euler equations (see \cite{LionsBook1} for this notion, and \cite{BardTiti} for its extension to the initial boundary value problem). The main
advantage of such a formulation is that a) the definition of dissipative solutions corresponds exactly with the limiting form of the relative entropy 
inequality (\ref{RelEntr<}) as $\eps\to 0$, and b) for any initial velocity field with finite kinetic energy that is solenoidal and tangential on $\d\Om$, 
there exists at least one dissipative solution of the incompressible Euler equations that is defined for all times --- see Proposition 4.2 on p. 156 
in \cite{LionsBook1}. Besides, should the Euler equations have a classical solution, it is known that all dissipative solutions with the same initial 
data must coincide with that classical solution --- see Proposition 4.1 on p. 155 in \cite{LionsBook1}. We have refrained from using the notion of 
dissipative solutions in the present paper for the sake of simplicity, and refer the interested reader to \cite{BGP} where the result analogous to 
Theorem \ref{T-Boltz>Eul} in the special case of Maxwell's accomodation condition is formulated in terms of dissipative solutions of the
incompressible Euler equations.

\bigskip
\noindent
\textbf{Acknowledgments.} The author is grateful to Profs. K. Aoki, C. Bardos T.-P. Liu and M. Sammartino for several useful suggestions and 
remarks on the material presented in this paper.


\end{document}